\documentclass[preprint,12pt]{elsarticle}

\usepackage{graphicx}
\usepackage[T1]{fontenc}
\usepackage[utf8]{inputenc}
\usepackage{textcomp}
\usepackage[]{amsmath, amssymb}
\usepackage{fancyvrb}
\usepackage{caption}
\usepackage{float}
\usepackage{url}
\usepackage{pdfpages}
\usepackage{listings}
\usepackage{amsmath, amssymb}
\usepackage{amsthm}
\usepackage{grffile}
\usepackage{multirow}
\usepackage{pgfplots}
\usepackage{silence}
\usepackage{subcaption}
\usepackage{hyperref}

\usepackage[colorinlistoftodos,prependcaption,textsize=tiny]{todonotes}
\usepackage{xargs}
\newcommandx{\unsure}[2][1=]{\todo[linecolor=red,backgroundcolor=red!25,bordercolor=red,#1]{#2}}
\newcommandx{\change}[2][1=]{\todo[linecolor=blue,backgroundcolor=blue!25,bordercolor=blue,#1]{#2}}
\newcommandx{\info}[2][1=]{\todo[linecolor=OliveGreen,backgroundcolor=OliveGreen!25,bordercolor=OliveGreen,#1]{#2}}
\newcommandx{\improvement}[2][1=]{\todo[linecolor=Plum,backgroundcolor=Plum!25,bordercolor=Plum,#1]{#2}}
\newcommandx{\thiswillnotshow}[2][1=]{\todo[disable,#1]{#2}}

\DefineVerbatimEnvironment{code}{Verbatim}{fontsize=\small, xleftmargin=5mm}
\newcommand*{\Scaleeq}[2][4]{\scalebox{#1}{$#2$}}%

\newcommand{\R}{{\mathbb R}}

\newcommand{\norm}[1]{\|#1\|}
\newcommand{\snorm}[1]{\|#1\|^2}

\newcommand{\der}[2]{\frac{d #1}{d #2}}

\newcommand{\Hi}{H^{-1}}

\definecolor{light-gray}{gray}{0.8}

\theoremstyle{plain}
\newtheorem{theorem}{Theorem}
\newtheorem{lemma}[theorem]{Lemma}

\theoremstyle{definition}
\newtheorem{defi}{Definition}

\theoremstyle{remark}

\theoremstyle{remark}
\newtheorem{remark}{Remark}

\newlength\figureheight
\newlength\figurewidth

\journal{Journal of Computational Physics}

\begin{document}

\begin{frontmatter}

	\title{Boundary and interface methods for energy stable finite difference discretizations of the dynamic beam equation}

	\author{Gustav Eriksson\corref{cor1}}
	\author{Jonatan Werpers}
	\author{David Niemelä}
	\author{Niklas Wik}
	\author{Valter Zethrin}
	\author{Ken Mattsson}

	\cortext[cor1]{Corresponding author \\ E-mail address: \texttt{gustav.eriksson@it.uu.se}}
	\address{Department of Information Technology, Uppsala University, PO Box 337, S-751 05 Uppsala, Sweden}

	\begin{abstract}
		We consider energy stable summation by parts finite difference methods (SBP-FD) for the homogeneous and piecewise homogeneous dynamic beam equation (DBE). Previously the constant coefficient problem has been solved with SBP-FD together with penalty terms (SBP-SAT) to impose boundary conditions. In this work we revisit this problem and compare SBP-SAT to the projection method (SBP-P). We also consider the DBE with discontinuous coefficients and present novel SBP-SAT, SBP-P and hybrid SBP-SAT-P discretizations for imposing interface conditions. Numerical experiments show that all methods considered are similar in terms of accuracy, but that SBP-P can be more computationally efficient (less restrictive time step requirement for explicit time integration methods) for both the constant and piecewise constant coefficient problems.
	\end{abstract}

	\begin{keyword}
		dynamic beam equation \sep boundary treatment \sep summation by parts \sep finite differences \sep high order methods
	\end{keyword}
\end{frontmatter}
\section{Introduction}
\label{sec: intro}
The dynamic beam equation (DBE) is a standard beam theory model describing the motion of free vibrations of a Euler-Bernoulli beam. The derivation of the equation dates back to the 18th century but it is still used in engineering applications today. For example, in construction of infrastructure involving beams such as buildings, bridges and railways.

Most of the literature on the DBE focus on solving the constant coefficient problem in the frequency domain. By taking the Laplace transform of \eqref{eq: cont_DBE} one can, for specific boundary conditions and external loads, derive the fundamental frequencies of the beam and thus obtain explicit analytical solutions \cite{LEE2019191,Goncalves2007,STIMACRONCEVIC2019118,YU2018571}. Although this approach is highly efficient when applicable, the analysis becomes significantly more complex when discontinuous beam parameters and general dynamic loads are required. An alternative is to utilize numerical methods. In \cite{Ansari2013} the spatial component of the solution is discretized with compact finite differences and used to numerically derive the natural frequencies of nanobeams. However, this approach assumes that the spatial and temporal parts of the solution function can be separated. In this paper we solve the DBE numerically in the time-domain using the \emph{method of lines}, i.e. we discretize the partial differential equation (PDE) in space and then solve the resulting system of ordinary differential equations (ODE). The spatial discretization is done using high order explicit finite difference methods. This approach is not limited by any specific choice of boundary condition or external loads, and does not assume that the solution takes any particular form (e.g. separate temporal and spatial components).

It has for a long time been known that non-trivial boundary procedures are required to obtain stable high order finite difference discretizations of initial-boundary value problems (IBVP) \cite{DeRango2001,Hesthaven1997,Bayliss1986,Abarbanel1997,Lele1992}. One way to manage this is to combine \emph{summation by parts} (SBP) finite difference operators with either the \emph{simultaneous approximation term} (SAT) method or the \emph{projection} (P) method to impose boundary conditions. With both SBP-SAT or SBP-P a semi-discrete energy estimate that mimics the continuous estimate can be derived, ensuring stability of the ODE system. The SBP-SAT and SBP-P methods have been applied to various PDE:s in the past, see for example \cite{Hicken2012,Mattsson2006, Mattsson2018,Almquist2019,Mattsson2018a}. See also \cite{Mattsson2015,Mattsson2016,Mattsson2014} for examples of the SBP-SAT method applied to IBVP:s involving third and fourth derivatives. In \cite{Mattsson2015} the SBP-SAT method was used to solve the DBE with constant coefficients.

In this paper we begin by deriving stable SBP-P discretizations of the DBE with clamped and free boundary conditions, and compare them to previously presented SBP-SAT discretizations \cite{Mattsson2015}. We then consider the DBE with piecewise constant coefficients, requiring additional interface conditions to couple the solution across the discontinuities. We present novel SBP-SAT, SBP-P and hybrid SBP-SAT-P discretizations for this problem and compare them numerically in terms of accuracy and computational efficiency. Previously SBP-SAT and SBP-P have been used to impose interface conditions for PDE:s with second derivatives in space, see for example \cite{Mattsson2006,Almquist2019,Mattsson2008}. Here we propose a novel third alternative utilizing SAT and projection simultaneously, SBP-SAT-P, where some interface conditions are imposed using SAT and others using projection. The focus of the current work is twofold: a): to develop a stable method for solving the DBE with general parameters, boundary conditions and external loads and b): to compare SBP-SAT, SBP-P and SBP-SAT-P for a problem involving high spatial derivatives (for which it is traditionally difficult to derive energy stable schemes).

This paper is organized as follows: In Section \ref{sec: probdesc} the DBE is introduced. In Section \ref{sec: defs} necessary definitions are presented. In Sections \ref{sec: hom_beam} and \ref{sec: p_hom_beam} the DBE with constant and piecewise constant coefficients respectively is analyzed. In Section \ref{sec: time_int} the time stepping scheme is presented. In Section \ref{sec: num_exp} the methods are compared and the analysis is verified numerically. Conclusions are drawn in Section \ref{sec: concs}.
\section{The dynamic beam equation}
\label{sec: probdesc}
Denote by $w(x,t)$ the beam's deflection from the $x$-axis, the DBE is then given by
\begin{equation}
	\Scaleeq[0.9]{
		\label{eq: cont_DBE}
		\begin{alignedat}{5}
			&\mu(x) \frac{\partial^2 w(x,t)}{\partial t^2} = - \frac{\partial^2}{\partial x^2} \left ( E(x) I(x) \frac{\partial^2 w(x,t)}{\partial x^2}  \right ) + q(x,t), \quad &&x_l \leq x \leq x_r, \quad &&&t > 0, \\
			&w(x,t) = f_1(x), \quad \frac{\partial w(x,t)}{\partial t} = f_2(x),  &&x_l \leq x \leq x_r,  &&&t = 0, \\
			&B_L(w) = \mathbf{G}_L(t), &&x = x_l, &&&t > 0,\\
			&B_R(w) = \mathbf{G}_R(t), &&x = x_r, &&&t > 0,
		\end{alignedat}
	}
\end{equation}
where $E(x)$ is the elastic modulus of the beam, $I(x)$ the second moment of area of the beam's cross-section, $\mu(x)$ the mass per unit length and $q(x,t)$ the external load. The initial data is $f_{1,2}(x)$. The boundary operators $B_{L,R}$ and boundary data $\mathbf{G}_{L,R}(t)$ determine the boundary conditions. For notational clarity, we rewrite the PDE as
\begin{equation}
	\label{eq: cont_DBE_math}
	b(x) u_{tt} = -(a(x) u_{xx})_{xx} + F(x,t),
\end{equation}
where subscripts denote partial differentiation. Note that the functions $a(x) = E(x) I(x)$ and $b(x) = \mu(x)$ are positive. Throughout this paper we will assume that the boundary data $\mathbf{G}_{L,R}(t)$ and forcing function $F(x,t)$ are zero (all methods considered are equally applicable with general non-homogeneous data).
\section{Definitions}
\label{sec: defs}
Let the inner product of two real-valued functions $u,v \in L^2[x_l,x_r]$ be defined by $(u,v) = \int _{x_l} ^{x_r} uv \, dx$ and the corresponding norm by $\snorm{u} = (u,u)$.

The domain is discretized into $m$ equidistant grid points given by
\begin{equation}
	x_i = x_l + (i-1)h, \quad i = 1,2,..., m, \quad h = \frac{x_r - x_l}{m-1}.
\end{equation}
A semi-discrete solution vector is given by $v = [v_1, v_2, ..., v_m]^\top$, where $v_i$ is the approximate solution at grid point $x_i$. Let $(u,v)_H = u^\top H v$, where $H = H^\top > 0$, define a discrete inner product for discrete real-valued vectors $u,v \in \R^m$, and $\snorm{v}_H = v^\top H v$ the corresponding norm. We also define
\begin{equation}
	\bar H = \begin{bmatrix}
		H & 0 \\ 0 & H
	\end{bmatrix}.
\end{equation}
The SBP operators used here are central finite difference operators with boundary closures carefully designed to mimic integration-by-parts in the semi-discrete setting. In this paper we present results for the 2nd, 4th and 6th order accurate fourth-derivative SBP operators presented in \cite{Mattsson2014}, the following definition is central:
\begin{defi}
	\label{def: D4_def}
	A difference operator
	\newline$D_4 = \Hi \left(N + e_l d_{3;l}^\top + d_{1;l} d_{2;l}^\top + e_r d_{3;r}^\top - d_{1;r} d_{2;r}^\top \right)$ approximating $\partial^4/\partial x^4$, using a 2pth-order accurate narrow-stencil in the interior, is said to be a 2pth-order diagonal-norm fourth-derivative SBP operator if $H = H^\top > 0$ is diagonal, $N = N^T \geq 0$, $d_{1;l}^\top v \approx -u_x |^l$, $d_{1;r}^\top v \approx u_x |^r$, $d_{2;l}^\top v \approx -u_{xx} |^l$, $d_{2;r}^\top v \approx u_{xx}|^r$, $d_{3;l}^\top v \approx -u_{xxx} |^l$ and $d_{3;r}^\top v \approx u_{xxx}|^r$ are finite difference approximations of the first, second and third normal derivatives at the left and right boundary points.
\end{defi}
The matrix $N$ can be decomposed in the following way:
\begin{equation}
	\label{eq: N_split}
	N =
	\tilde N
	+ h   \alpha_{II}  \left( d_{2;l}d_{2;l}^\top + d_{2;r}d_{2;r}^\top \right)
	+ h^3 \alpha_{III} \left( d_{3;l}d_{3;l}^\top + d_{3;r}d_{3;r}^\top \right),
\end{equation}
where $\tilde N = \tilde N^\top \geq 0$ and $\alpha_{II}$ and $\alpha_{III}$ are positive constants not dependent on $h$. In Figure \ref{fig: alpha_zones} the values ensuring positive semi-definiteness of $\tilde N$ are presented for the 2nd, 4th and 6th order accurate SBP operators used here.
\begin{figure}[htbp]
	\centering
	\includegraphics[width=\textwidth, clip]{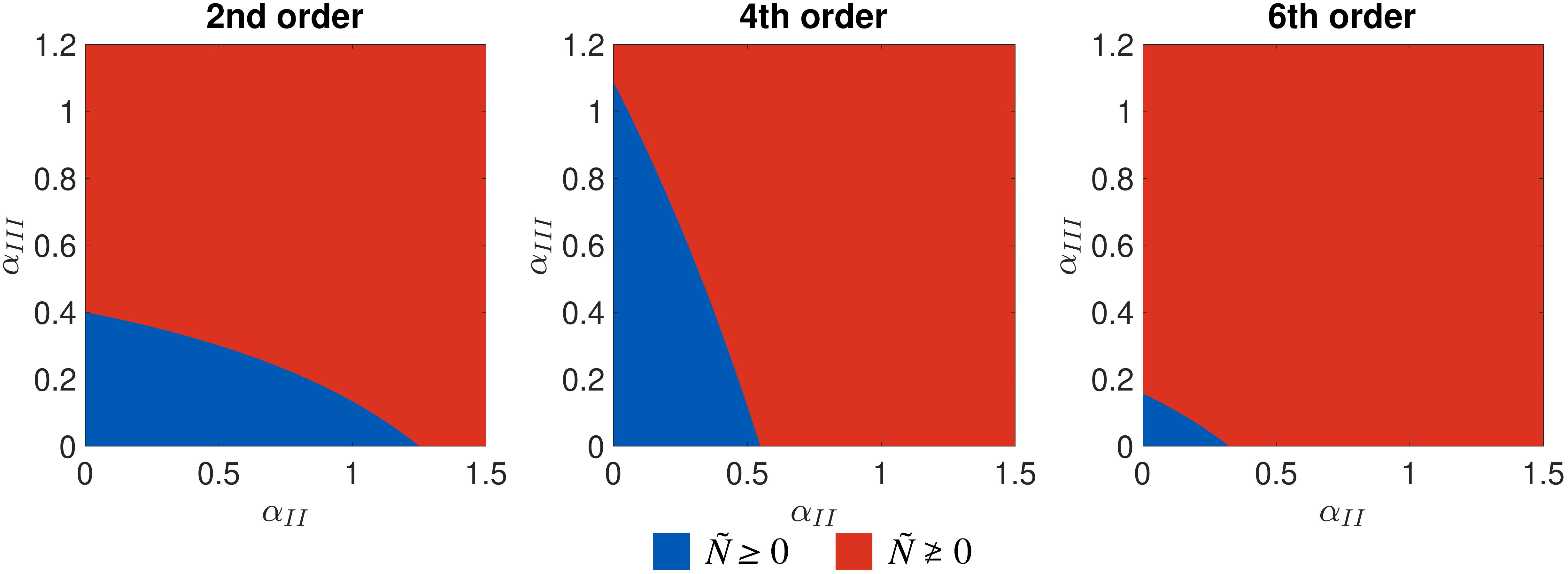}
	\caption{Influence of parameters $\alpha_{II}$ and $\alpha_{III}$ on definiteness of $\tilde N$.}
	\label{fig: alpha_zones}
\end{figure}
\section{Homogeneous beam}
\label{sec: hom_beam}
\subsection{Well-posedness}
We begin by considering a homogeneous beam, i.e. $a(x)$ and $b(x)$ are constant functions. Multiplying \eqref{eq: cont_DBE_math} by $u_t$, integrating over the domain and using integration by parts (the energy method) leads to
\begin{equation}
	\label{eq: cont_en_est}
	\frac{d}{dt} E = 2a[-u_t u_{xxx} + u_{tx} u_{xx} ] _{x_l} ^{x_r},
\end{equation}
where $E$ is an energy defined as
\begin{equation}
	E = b \norm{u_t}^2 + a \snorm{u_{xx}}.
\end{equation}
Assuming existence, well-posedness of \eqref{eq: cont_DBE_math} is achieved by imposing the minimum number of boundary conditions such that the right-hand side in \eqref{eq: cont_en_est} is non-positive with zero boundary data, i.e. the energy for the homogeneous problem is non-increasing. For this problem the minimum number of boundary conditions are two on each boundary \cite{Mattsson2015}. In this paper we consider the following homogeneous boundary conditions:
\begin{equation}
	\label{eq: cont_BC}
	\begin{aligned}
		 & \text{\it{clamped:}} \quad &  & u = 0,  \quad &  &  & u_{x} = 0, \quad &  &  &  & x = x_{l,r}, \\
		 & \text{\it{free:}}          &  & u_{xx} = 0,   &  &  & u_{xxx} = 0,     &  &  &  & x = x_{l,r}.
	\end{aligned}
\end{equation}
Either combination of these lead to energy conservation,
\begin{equation}
	\frac{d}{dt} E = 0.
\end{equation}
\subsection{Spatial discretization}
A semi-discrete SBP discretization of \eqref{eq: cont_DBE_math} with constant $a$ and $b$ is
\begin{equation}
	\label{eq: disc_single_dbe}
	b v_{tt} = -a D_4 v.
\end{equation}
The semi-discrete formulations of the boundary conditions \eqref{eq: cont_BC} are
\begin{equation}
	\label{eq: disc_BC}
	\begin{aligned}
		 & \text{\it{clamped:}} \quad &  & e_{l}^\top v = 0,  \quad   &  &  & d_{1;l}^\top v = 0, \quad &  & e_{r}^\top v = 0,  \quad   &  &  & d_{1;r}^\top v = 0, \\
		 & \text{\it{free:}}          &  & d_{2;l}^\top v = 0,  \quad &  &  & d_{3;l}^\top v = 0, \quad &  & d_{2;r}^\top v = 0,  \quad &  &  & d_{3;r}^\top v = 0.
	\end{aligned}
\end{equation}
To impose the boundary conditions while simultaneously obtaining a stable ODE, we consider two methods: SAT and projection. With SAT the ODE system takes the form
\begin{equation}
	\label{eq: disc_single_SAT}
	b v_{tt} = -a D_4 v + SAT,
\end{equation}
where $SAT = a H^{-1} Bv$ and $B$ is a carefully chosen term so that the boundary conditions are imposed and an energy estimate is obtained, i.e. the eigenvalues of the matrix $-a D_4 + a H^{-1} B$ are real and non-positive. Stable SBP-SAT discretizations of the DBE with clamped and free boundary conditions are derived in \cite{Mattsson2015}. To make the paper self-contained, we also present them here.

The projection method augments the system \eqref{eq: disc_single_dbe} as
\begin{equation}
	\label{eq: disc_single_proj}
	b v_{tt} = -a P D_4 P v,
\end{equation}
where $P$ is a projection from $\mathbb{R}^m$ onto a subspace where the boundary conditions are satisfied. If $P$ is an orthogonal projection with respect to the inner product $(\cdot,\cdot)_H$, i.e.
\begin{equation}
	\label{eq: proj_self_adj}
	(Pu,v)_H = (u,Pv)_H, \quad u,v \in \R^m,
\end{equation}
the energy method can be used to show that the matrix $-a P D_4 P$ only has real and non-positive eigenvalues. For the derivation of $P$ and examples on using the projection method for imposing boundary conditions, see \cite{Mattsson2006,Mattsson2018,Olsson1995a,Olsson1995}.

In Sections \ref{sec: sat_single_clamped}-\ref{sec: proj_single_free} we derive semi-discrete energy estimates for \eqref{eq: disc_single_dbe} with boundary conditions \eqref{eq: disc_BC} imposed using SAT and projection. To make the analysis more readable, we only present boundary terms corresponding to the left boundary. The right boundary terms are done analogously.
\begin{remark}
	In this work we only consider homogeneous boundary conditions. For inhomogeneous boundary conditions additional time dependent data terms are included in \eqref{eq: disc_single_SAT} and \eqref{eq: disc_single_proj}. Boundary and forcing data do not affect stability.
\end{remark}
\subsection{Clamped boundary conditions with SAT}
\label{sec: sat_single_clamped}
SAT for clamped boundary conditions, first presented in Lemma 4.3 in \cite{Mattsson2015}, is given by
\begin{equation}
	\label{eq: disc_SAT_clamped}
	SAT_c = -a \Hi (d_{3;l} + h^{-3} \tau_l e_l) e_l^\top v - a \Hi (d_{2;l} + h^{-1} \sigma_l d_{1;l}) d_{1;l}^\top v,
\end{equation}
where $\tau_{l}$ and $\sigma_{l}$ are constants chosen for stability. For completeness, we restate the Lemma and stability proof here.
\begin{lemma}
	The system \eqref{eq: disc_single_SAT} is a stable approximation of the DBE with clamped boundary conditions if SAT is given by \eqref{eq: disc_SAT_clamped} and
	\begin{equation}
		\tau_{l} = \frac{1}{\alpha_{III}} \quad \text{and} \quad \sigma_{l} = \frac{1}{\alpha_{II}}.
	\end{equation}
\end{lemma}
\begin{proof}
	The energy method applied to \eqref{eq: disc_single_SAT} with SAT given by \eqref{eq: disc_SAT_clamped} together with \eqref{eq: N_split} leads to
	\begin{equation}
		\label{eq: disc_clamp_en_est}
		\der{}{t} E = 0,
	\end{equation}
	where
	\begin{equation}
		E = b \snorm{v_t}_H + a (v^\top \tilde N v + w^\top A w),
	\end{equation}
	\begin{equation}
		w = \begin{bmatrix}
			e_{l}^\top v   \\
			d_{3;l}^\top v \\
			d_{1;l}^\top v \\
			d_{2;l}^\top v
		\end{bmatrix} \quad \text{and} \quad
		A = \begin{bmatrix}
			h^{-3} \tau_{l} & 1                & 0                 & 0             \\
			1               & h^3 \alpha_{III} & 0                 & 0             \\
			0               & 0                & h^{-1} \sigma_{l} & 1             \\
			0               & 0                & 1                 & h \alpha_{II}
		\end{bmatrix}.
	\end{equation}
	For $E$ to be a valid energy it must be positive semi-definite. The matrix $\tilde N$ is positive semi-definite for specific choices of $\alpha_{II}$ and $\alpha_{III}$, see Figure \ref{fig: alpha_zones}. It remains to show which $\tau_l$ and $\sigma_l$ leads to positive semi-definite $A$. By applying Sylvester's criterion we get the conditions
	\begin{equation}
		\tau_{l} = \frac{1}{\alpha_{III}} \quad \text{and} \quad \sigma_{l} = \frac{1}{\alpha_{II}}.
	\end{equation}
\end{proof}
Note that multiple choices of $\alpha_{II}$ and $\alpha_{III}$ lead to stability.
\subsection{Clamped boundary conditions with projection}
\label{sec: proj_single_clamp}
The semi-discrete clamped boundary conditions can be written as $L_c v = 0$ where
\begin{equation}
	L_c = \begin{bmatrix}
		e_l^\top \\
		d_{1;l}^\top
	\end{bmatrix}.
\end{equation}
The projection operator is given by
\begin{equation}
	\label{eq: P_clamp}
	P = I_m - \Hi L_c^\top (L_c \Hi L_c^\top)^{-1} L_c,
\end{equation}
where $I$ is the $m \times m$ identity matrix. The following Lemma is one of the main results in this work:
\begin{lemma}
	The system \eqref{eq: disc_single_proj} is a stable approximation of the DBE with clamped boundary conditions if the projection operator P is given by \eqref{eq: P_clamp}.
\end{lemma}
\begin{proof}
	The energy method applied to \eqref{eq: disc_single_proj} leads to
	\begin{equation}
		b(v_t,v_{tt})_H = -a (P v_t, D_4 P v)_H,
	\end{equation}
	where the self-adjoint property \eqref{eq: proj_self_adj} is used. Let $\tilde v = P v$ denote the projected solution vector and use Definition \ref{def: D4_def}, we get
	\begin{equation}
		\label{eq: proj_clamped_en_est}
		\der{}{t} E = -2 a (e_l^\top \tilde v_t d_{3;l}^\top \tilde v + d_{1;l}^\top \tilde v_t d_{2;l}^\top \tilde v),
	\end{equation}
	where
	\begin{equation}
		E = b \snorm{v_t}_H + a \tilde v^\top N \tilde v
	\end{equation}
	defines an energy. Since $L_c \tilde v = L_c P v = 0$ we have $e_l^\top \tilde v = d_{1;l}^\top \tilde v = 0$ and thus
	\begin{equation}
		\der{}{t} E = 0.
	\end{equation}
\end{proof}
\subsection{Free boundary conditions with SAT}
\label{sec: sat_single_free}
SAT for free boundary conditions, first presented in Lemma 4.2 in \cite{Mattsson2015}, are given by
\begin{equation}
	\label{eq: disc_SAT_free}
	SAT_f = a \tau_l \Hi d_{1;l} d_{2;l}^\top v + a \sigma_l \Hi e_l d_{3;l}^\top v.
\end{equation}
where $\tau_{l}$ and $\sigma_{l}$ are constants chosen for stability. For completeness, we restate the Lemma and stability proof here.
\begin{lemma}
	The system \eqref{eq: disc_single_SAT} is a stable approximation of the DBE with free boundary conditions if SAT is given by \eqref{eq: disc_SAT_free} and
	\begin{equation}
		\tau_l = \sigma_l = 1.
	\end{equation}
\end{lemma}
\begin{proof}
	The energy method applied to \eqref{eq: disc_single_SAT} with SAT given by \eqref{eq: disc_SAT_free} leads to
	\begin{equation}
		\der{}{t} E = 2a(\tau_l - 1) d_{1;l}^\top v_t d_{2;l}^\top v + 2a(\sigma_l - 1) e_l^\top v_t d_{3;l}^\top v,
	\end{equation}
	where
	\begin{equation}
		E = b \snorm{v_t}_H + a v^\top N v
	\end{equation}
	defines an energy. The choice $\tau_l = \sigma_l = 1$ results in
	\begin{equation}
		\der{}{t} E = 0.
	\end{equation}
\end{proof}
\subsection{Free boundary conditions with projection}
\label{sec: proj_single_free}
The semi-discrete free boundary conditions can be written as $L_f v = 0$ where
\begin{equation}
	L_f = \begin{bmatrix}
		d_{2;l}^\top \\
		d_{3;l}^\top
	\end{bmatrix}.
\end{equation}
The projection operator is given by
\begin{equation}
	\label{eq: P_free}
	P = I_m - \Hi L_f^\top (L_f \Hi L_f^\top)^{-1} L_f,
\end{equation}
where $I$ is the $m \times m$ identity matrix. The following Lemma is one of the main results in this work:
\begin{lemma}
	The system \eqref{eq: disc_single_proj} is a stable approximation of the DBE with free boundary conditions if the projection operator P is given by \eqref{eq: P_free}.
\end{lemma}
\begin{proof}
	The energy method applied to \eqref{eq: disc_single_proj} leads to
	\begin{equation}
		b(v_t,v_{tt})_H = -a (P v_t, D_4 P v)_H,
	\end{equation}
	where the self-adjoint property \eqref{eq: proj_self_adj} is used. Let $\tilde v = P v$ denote the projected solution vector and use Definition \ref{def: D4_def}, we get
	\begin{equation}
		\label{eq: proj_free_en_est}
		\der{}{t} E = -2 a (e_l^\top \tilde v_t d_{3;l}^\top \tilde v + d_{1;l}^\top \tilde v_t d_{2;l}^\top \tilde v),
	\end{equation}
	where
	\begin{equation}
		E = b \snorm{v_t}_H + a \tilde v^\top N \tilde v,
	\end{equation}
	defines an energy. Since $L_f \tilde v = L_f P v = 0$ we have $d_{2;l}^\top \tilde v = d_{3;l}^\top \tilde v = 0$ and thus
	\begin{equation}
		\der{}{t} E = 0.
	\end{equation}
\end{proof}
\section{Piecewise homogeneous beam}
In this section we turn to the DBE with piecewise constant coefficients. In the analysis, we assume that $x_l < 0$ and $x_r > 0$ and consider \eqref{eq: cont_DBE_math} with $a(x)$ and $b(x)$ piecewise constant with a discontinuity at $x = 0$, i.e.
\begin{equation}
	\label{eq: discont_a_b}
	a(x)=
	\begin{cases}
		a^{(1)}, & \text{if } x < 0 \\
		a^{(2)}, & \text{if } x > 0 \\
	\end{cases}, \quad
	b(x)=
	\begin{cases}
		b^{(1)}, & \text{if } x < 0 \\
		b^{(2)}, & \text{if } x > 0 \\
	\end{cases},
\end{equation}
where $a^{(1,2)}$ and $b^{(1,2)}$ are positive constants. We begin by deriving well-posed interface conditions in the continuous setting, we then discretize the equations in space and derive stable semi-discrete formulations with the interface conditions imposed using SAT, projection and hybrid SAT and projection.
\label{sec: p_hom_beam}
\subsection{Well-posedness}
The energy method applied to \eqref{eq: cont_DBE_math} with $a(x)$ and $b(x)$ given by \eqref{eq: discont_a_b} leads to
\begin{equation}
	\frac{d}{dt} E^{(a,b)} = a^{(1)} [u_{tx} u_{xx} - u_t u_{xxx}]_{x_l} ^{0^-} + a^{(2)} [u_{tx} u_{xx} - u_t u_{xxx}]_{0^+} ^{x_r},
\end{equation}
where
\begin{equation}
	E^{(a,b)} = \int _{x_l} ^{0^-} b^{(1)} u_t^2 + a^{(1)} u_{xx}^2 \: dx + \int _{0^+} ^{x_r} b^{(2)} u_t^2 + a^{(2)} u_{xx}^2 \: dx,
\end{equation}
defines an energy. To obtain an energy estimate we split $u(x,t)$ at the interface such that
\begin{equation}
	u(x,t){} = \begin{cases}
		u^{(1)}(x,t), & \text{if } x < 0 \\
		u^{(2)}(x,t), & \text{if } x > 0 \\
	\end{cases},
\end{equation}
and rewrite \eqref{eq: cont_DBE_math} as a system of PDEs:
\begin{equation}
	\label{eq: cont_piece_system}
	\begin{aligned}
		b^{(1)} u^{(1)}_{tt} & = -a^{(1)} u^{(1)}_{xxxx}, \quad &  & x_l \leq x \leq 0, &  & t > 0, \\
		b^{(2)} u^{(2)}_{tt} & = -a^{(2)} u^{(2)}_{xxxx},       &  & 0 \leq x \leq x_r, &  & t > 0.
	\end{aligned}
\end{equation}
The energy estimate becomes
\begin{equation}
	\label{eq: cont_system_enest}
	\frac{d}{dt} E^{(a,b)} = a^{(1)} [u^{(1)}_{tx} u^{(1)}_{xx} - u^{(1)}_t u^{(1)}_{xxx}]_{x_l} ^{0} + a^{(2)} [u^{(2)}_{tx} u^{(2)}_{xx} - u^{(2)}_t u^{(2)}_{xxx}]_0 ^{x_r},
\end{equation}
where
\begin{equation}
	\label{eq: cont_distcont_energy}
	E^{(a,b)} = b^{(1)} \norm{u^{(1)}_t}^2 + b^{(2)} \norm{u^{(2)}_t}^2 + a^{(1)} \snorm{u^{(1)}_{xx}} + a^{(2)} \snorm{u^{(2)}_{xx}}.
\end{equation}
In addition to outer boundary conditions, the following \emph{interface conditions} will bound the energy:
\begin{equation}
	\label{eq: cont_int_cond}
	\begin{aligned}
		u^{(1)}               & =        u^{(2)},        & x & = 0, \\
		u_x^{(1)}             & = u_x^{(2)},             & x & = 0, \\
		a^{(1)} u_{xx}^{(1)}  & = a^{(2)} u_{xx}^{(2)},  & x & = 0, \\
		a^{(1)} u_{xxx}^{(1)} & = a^{(2)} u_{xxx}^{(2)}, & x & = 0. \\
	\end{aligned}
\end{equation}
With either boundary condition in (\ref{eq: cont_BC}) and the interface conditions (\ref{eq: cont_int_cond}) the energy given by (\ref{eq: cont_distcont_energy}) is conserved.
\subsection{Spatial discretization} 
Let $v^{(1)}$ and $v^{(2)}$ denote solution vectors on the left and right side of the discontinuity respectively. For notational clarity we assume that the same SBP operators are used in both blocks. A consistent semi-discrete formulation of \eqref{eq: cont_piece_system} is given by
\begin{equation}
	\label{eq: disc_system_ode}
	\begin{aligned}
		(B \otimes I_m) w_{tt} = -(A \otimes D_4) w,
	\end{aligned}
\end{equation}
where
\begin{equation}
	\nonumber
	w = \begin{bmatrix}
		v^{(1)} \\ v^{(2)}
	\end{bmatrix}, \quad
	B = \begin{bmatrix}
		b^{(1)} & 0 \\ 0 & b^{(2)}
	\end{bmatrix}, \quad
	A = \begin{bmatrix}
		a^{(1)} & 0 \\ 0 & a^{(2)}
	\end{bmatrix},
\end{equation}
and $\otimes$ denotes the Kronecker product. The semi-discrete formulation of the interface conditions \eqref{eq: cont_int_cond} is
\begin{equation}
	\label{eq: disc_inter_conds}
	\begin{aligned}
		e_r^\top v^{(1)} - e_l^\top v^{(2)}                         & = 0            , \\
		d_{1;r}^\top v^{(1)} + d_{1;l}^\top v^{(2)}                 & = 0,             \\
		a^{(1)} d_{2;r}^\top v^{(1)} + a^{(2)} d_{2;l}^\top v^{(2)} & = 0,             \\
		a^{(1)} d_{3;r}^\top v^{(1)} + a^{(2)} d_{3;l}^\top v^{(2)} & = 0.             \\
	\end{aligned}
\end{equation}
In the analysis of the semi-discrete multiblock problem we assume that outer boundary conditions are imposed correctly (for example by using any of the methods discussed in Sections \ref{sec: sat_single_clamped}-\ref{sec: proj_single_free}). Therefore, the terms corresponding to outer boundaries are not included in the equations presented.

As for the single-block problem the semi-discrete interface conditions \eqref{eq: disc_inter_conds} should be imposed such that an energy estimate is obtained. With SBP-SAT additional terms are added to \eqref{eq: disc_system_ode}. With SBP-P the system \eqref{eq: disc_system_ode} is augmented similarly as \eqref{eq: disc_single_proj}, using a projection onto the subspace of $\mathbb{R}^{2m}$ where the interface conditions \eqref{eq: disc_inter_conds} are fulfilled. Since the dual-block projection operator acts on the full solution vector $w$, the self-adjoint property becomes
\begin{equation}
	\label{eq: proj_self_adj_system}
	(Pu,v)_{\bar H} = (u,Pv)_{\bar H}, \quad u,v \in \R^{2m}.
\end{equation}
With SBP-SAT-P both SAT and projection are used to modify the system simultaneously. In Sections \ref{sec: proj_multiblock_sat}-\ref{sec: proj_multiblock_hybrid} semi-discrete formulations and stability proofs of each of the three methods are presented.
\subsection{Interface conditions with SAT}
\label{sec: proj_multiblock_sat}
A consistent semi-discrete approximation of \eqref{eq: cont_piece_system} with the interface conditions \eqref{eq: cont_int_cond} imposed using the SAT method is given by
\begin{equation}
	\label{eq: disc_system_ode_sat}
	(B \otimes I_m) w_{tt} = -(A \otimes D_4) w + SAT,
\end{equation}
where $SAT = SAT_1 + SAT_2 + SAT_3 + SAT_4$,
\begin{equation}
	\begin{alignedat}{1}
		SAT_1 &= \begin{bmatrix}
			- \Hi (\frac{\tau}{h^3} e_r + \frac{a^{(1)}}{2} d_{3;r}) e_r^\top & \Hi (\frac{\tau}{h^3} e_r + \frac{a^{(1)}}{2} d_{3;r}) e_l^\top  \\
			\Hi (\frac{\tau}{h^3} e_l + \frac{a^{(2)}}{2} d_{3;l}) e_r^\top   & -\Hi (\frac{\tau}{h^3} e_l + \frac{a^{(2)}}{2} d_{3;l}) e_l^\top
		\end{bmatrix} w, \\
		SAT_2 &= \begin{bmatrix}
			- \Hi (\frac{\sigma}{h} d_{1;r} - \frac{a^{(1)}}{2} d_{2;r}) d_{1;r}^\top & - \Hi (\frac{\sigma}{h} d_{1;r} - \frac{a^{(1)}}{2} d_{2;r}) d_{1;l}^\top \\
			- \Hi (\frac{\sigma}{h} d_{1;l} + \frac{a^{(2)}}{2} d_{2;l}) d_{1;r}^\top & - \Hi (\frac{\sigma}{h} d_{1;l} + \frac{a^{(2)}}{2} d_{2;l}) d_{1;l}^\top
		\end{bmatrix} w, \\
		SAT_3 &= \begin{bmatrix}
			- \frac{a^{(1)}}{2} \Hi d_{1;r} d_{2;r}^\top & - \frac{a^{(2)}}{2} \Hi d_{1;r} d_{2;l}^\top \\
			\frac{a^{(1)}}{2} \Hi d_{1;l} d_{2;r}^\top   & \frac{a^{(2)}}{2} \Hi d_{1;l} d_{2;l}^\top
		\end{bmatrix} w, \\
		SAT_4 &= \begin{bmatrix}
			\frac{a^{(1)}}{2} \Hi e_r d_{3;r}^\top & \frac{a^{(2)}}{2} \Hi e_r d_{3;l}^\top \\
			\frac{a^{(1)}}{2} \Hi e_l d_{3;r}^\top & \frac{a^{(2)}}{2} \Hi e_l d_{3;l}^\top
		\end{bmatrix} w,
	\end{alignedat}
\end{equation}
and $\tau$ and $\sigma$ are parameters tuned for stability. The ansatz for the SAT in \eqref{eq: disc_system_ode_sat} is similar in structure to the SAT presented in \cite{Mattsson2008} for imposing interface conditions for the second order wave equation. The following Lemma is one of the main results of this paper:
\begin{lemma}
	The system \eqref{eq: disc_system_ode_sat} is a stable approximation of the DBE with piecewise constant parameters if
	\begin{equation}
		\tau = \frac{1}{4\alpha_{III}} (a^{(1)} + a^{(2)}), \quad \sigma = \frac{1}{4\alpha_{II}} (a^{(1)} + a^{(2)}),
	\end{equation}
	and $\alpha_{II}$ and $\alpha_{III}$ are chosen so that $\tilde N$ is positive semi-definite.
\end{lemma}
\begin{proof}
	The energy method applied to \eqref{eq: disc_system_ode_sat} together with \eqref{eq: N_split} and Definition \ref{def: D4_def} leads to
	\begin{equation}
		\label{eq: disc_system_sat_enest1}
		\der{}{t} E = 0,
	\end{equation}
	where
	\begin{equation}
		\begin{alignedat}{1}
			E &= b^{(1)} \snorm{v^{(1)}_t}_H + b^{(2)} \snorm{v^{(2)}_t}_H + a^{(1)} ( v^{(1)})^\top \tilde N v^{(1)} + a^{(2)} ( v^{(2)})^\top \tilde N v^{(2)} \\
			&+ (w^{(1)})^\top A_1 w^{(1)} + (w^{(2)})^\top A_2 w^{(2)},
		\end{alignedat}
	\end{equation}
	and
	\begin{equation}
		\begin{alignedat}{3}
			w_1 &= \begin{bmatrix}
				e_r^\top v^{(1)}     \\
				d_{3;r}^\top v^{(1)} \\
				e_l^\top v^{(2)}     \\
				d_{3;l}^\top v^{(2)}
			\end{bmatrix}, \quad
			A_1 &&= \begin{bmatrix}
				\frac{\tau}{h^3}   & \frac{a^{(1)}}{2}        & \frac{-\tau}{h^3}  & \frac{-a^{(2)}}{2}       \\
				\frac{a^{(1)}}{2}  & a^{(1)} h^3 \alpha_{III} & \frac{-a^{(1)}}{2} & 0                        \\
				\frac{-\tau}{h^3}  & \frac{-a^{(1)}}{2}       & \frac{\tau}{h^3}   & \frac{a^{(2)}}{2}        \\
				\frac{-a^{(2)}}{2} & 0                        & \frac{a^{(2)}}{2}  & a^{(2)} h^3 \alpha_{III}
			\end{bmatrix}, \\
			w_2 &= \begin{bmatrix}
				d_{1;r}^\top v^{(1)} \\
				d_{2;r}^\top v^{(1)} \\
				d_{1;l}^\top v^{(2)} \\
				d_{2;l}^\top v^{(2)}
			\end{bmatrix}, \quad
			A_2 &&= \begin{bmatrix}
				\frac{\sigma}{h}   & \frac{-a^{(1)}}{2}    & \frac{\sigma}{h}   & \frac{a^{(2)}}{2}     \\
				\frac{-a^{(1)}}{2} & a^{(1)} h \alpha_{II} & \frac{-a^{(1)}}{2} & 0                     \\
				\frac{\sigma}{h}   & \frac{-a^{(1)}}{2}    & \frac{\sigma}{h}   & \frac{a^{(2)}}{2}     \\
				\frac{a^{(2)}}{2}  & 0                     & \frac{a^{(2)}}{2}  & a^{(2)} h \alpha_{II}
			\end{bmatrix}.
		\end{alignedat}
	\end{equation}
	For $E$ to be a valid energy it must be positive semi-definite. The matrix $\tilde N$ is positive semi-definite for specific choices of $\alpha_{II}$ and $\alpha_{III}$, see Figure \ref{fig: alpha_zones}. It remains to show which $\tau$ and $\sigma$ leads to positive semi-definite $A_1$ and $A_2$. By Sylvester's criterion we get the conditions
	\begin{equation}
		\label{eq: disc_multiblock_sat_params}
		\tau = \frac{1}{4\alpha_{III}} (a^{(1)} + a^{(2)}) \quad \text{and} \quad \sigma = \frac{1}{4\alpha_{II}} (a^{(1)} + a^{(2)}).
	\end{equation}
\end{proof}
Note that multiple choices of $\alpha_{II}$ and $\alpha_{III}$ lead to stability.
\subsection{Interface conditions with projection}
\label{sec: proj_multiblock_proj}
A consistent semi-discrete approximation of \eqref{eq: cont_piece_system} with the interface conditions \eqref{eq: cont_int_cond} imposed using the projection method is given by
\begin{equation}
	\label{eq: disc_system_ode_proj}
	(B \otimes I_m) w_{tt} = -P (A \otimes D_4) P w,
\end{equation}
where
\begin{equation}
	\label{eq: disc_syst_P}
	P = I_{2m} - \bar H^{-1} L^\top (L \bar H^{-1} L^\top)^{-1} L,
\end{equation}
$I_{2m}$ is the $2m \times 2m$ identity matrix and $L$ is given by
\begin{equation}
	L = \begin{bmatrix}
		e_r^\top             & - e_l^\top           \\
		d_{1;r}^\top         & d_{1;l}^\top         \\
		a^{(1)} d_{2;r}^\top & a^{(2)} d_{2;l}^\top \\
		a^{(1)} d_{3;r}^\top & a^{(2)} d_{3;l}^\top
	\end{bmatrix}.
\end{equation}
Note that $L w = 0$ defines the interface conditions \eqref{eq: disc_inter_conds}. The following Lemma is one of the main results of this paper:
\begin{lemma}
	The system \eqref{eq: disc_system_ode_proj} is a stable approximation of the DBE with piecewise constant parameters.
\end{lemma}
\begin{proof}
	The energy method applied to \eqref{eq: disc_system_ode_proj} leads to
	\begin{equation}
		(w_t,(B \otimes I_m)w_{tt})_{\bar H} = -(P w_t, (A \otimes D_4) P w)_{\bar H},
	\end{equation}
	where the self-adjoint property \eqref{eq: proj_self_adj_system} is used. Let $\tilde w = \begin{bmatrix}
			\tilde v^{(1)} \\ \tilde v^{(2)}
		\end{bmatrix} = P w$ denote the projected solution and use Definition \ref{def: D4_def}, we get
	\begin{equation}
		\label{eq: disc_system_enest_proj}
		\begin{aligned}
			\der{}{t} E & = -2 a^{(1)}(e_r^\top \tilde v^{(1)}_t d_{3;r}^\top \tilde v^{(1)} - d_{1;r}^\top \tilde v^{(1)}_t d_{2;r}^\top \tilde v^{(1)})  \\
			            & - 2 a^{(2)} (e_l^\top \tilde v^{(2)}_t d_{3;l}^\top \tilde v^{(2)} + d_{1;l}^\top \tilde v^{(2)}_t d_{2;l}^\top \tilde v^{(2)}),
		\end{aligned}
	\end{equation}{}
	where
	\begin{equation}
		E = b^{(1)} \snorm{v^{(1)}_t}_H + b^{(2)} \snorm{v^{(2)}_t}_H + a^{(1)} (\tilde v^{(1)})^\top N \tilde v^{(1)} + a^{(2)} (\tilde v^{(2)})^\top N \tilde v^{(2)},
	\end{equation}
	defines an energy. Since $L \tilde w = L P w = 0$, i.e. the interface conditions \eqref{eq: disc_inter_conds} are fulfilled for $\tilde v^{(1)}$ and $\tilde v^{(2)}$, the terms in the right-hand side of \eqref{eq: disc_system_enest_proj} cancel and we get
	\begin{equation}
		\der{}{t} E = 0.
	\end{equation}
\end{proof}
\subsection{Interface conditions with hybrid SAT and projection}
\label{sec: proj_multiblock_hybrid}
A consistent semi-discrete approximation of \eqref{eq: cont_piece_system} with the first two interface conditions in \eqref{eq: cont_int_cond} imposed using the projection method and last two using the SAT method is given by
\begin{equation}
	\label{eq: disc_system_ode_projhybrid}
	(B \otimes I_m) w_{tt} = -P (A \otimes D_4) P w + SAT,
\end{equation}
where $SAT = SAT_1 + SAT_2$,
\begin{equation}
	\label{eq: disc_projhybrid_sats}
	\begin{aligned}
		SAT_1 & = P \begin{bmatrix}
			-\frac{a^{(1)}}{2} \Hi d_{1;r} d_{2;r}^\top & -\frac{a^{(2)}}{2} \Hi d_{1;r} d_{2;l}^\top \\
			\frac{a^{(1)}}{2} \Hi d_{1;l} d_{2;r}^\top  & \frac{a^{(2)}}{2} \Hi d_{1;l} d_{2;l}^\top
		\end{bmatrix} P w, \\
		SAT_2 & = P \begin{bmatrix}
			\frac{a^{(1)}}{2} \Hi e_r d_{3;r}^\top & \frac{a^{(2)}}{2} \Hi e_r d_{3;l}^\top \\
			\frac{a^{(1)}}{2} \Hi e_l d_{3;r}^\top & \frac{a^{(2)}}{2} \Hi e_l d_{3;l}^\top
		\end{bmatrix} P w.
	\end{aligned}
\end{equation}
The projection operator is given by
\begin{equation}
	\label{eq: disc_syst_Phyb}
	P = I_{2m} - \bar H^{-1} L^\top (L \bar H^{-1} L^\top)^{-1} L,
\end{equation}
where
\begin{equation}
	L = \begin{bmatrix}
		e_r^\top     & - e_l^\top   \\
		d_{1;r}^\top & d_{1;l}^\top
	\end{bmatrix}.
\end{equation}
Note that since $P$ appears in \eqref{eq: disc_projhybrid_sats} the SAT are weakly imposing the interface conditions in the subspace of projected solutions. The following Lemma is one of the main results of this paper:
\begin{lemma}
	The system \eqref{eq: disc_system_ode_projhybrid} is a stable approximation of the DBE with piecewise constant parameters.
\end{lemma}
\begin{proof}
	The energy method applied to \eqref{eq: disc_system_ode_projhybrid} leads to
	\begin{equation}
		(w_t,(B \otimes I_m)w_{tt})_{\bar H} = -(w_t, P (A \otimes D_4) P w + SAT)_{\bar H}.
	\end{equation}
	Let $\tilde w = \begin{bmatrix}
			\tilde v^{(1)} \\ \tilde v^{(2)}
		\end{bmatrix} = P w$ denote the projected solution. Using the self-adjoint property \eqref{eq: proj_self_adj_system} and Definition \ref{def: D4_def} results in
	\begin{equation}
		\label{eq: disc_syst_projhyb_enest}
		\begin{alignedat}{2}
			\der{}{t} E &= a^{(1)} (e_l^\top \tilde v_t^{(2)} - e_r^\top \tilde v_t^{(1)}) d_{3;r}^\top \tilde v^{(1)} &&+ a^{(1)} (d_{1;l}^\top \tilde v_t^{(2)} + d_{1;r}^\top \tilde v_t^{(1)}) d_{2;r}^\top \tilde v^{(1)} \\
			&- a^{(2)} (e_l^\top \tilde v^{(2)}_t - e_r^\top \tilde v^{(1)}_t) d_{3;l}^\top \tilde v^{(2)} &&- a^{(2)} (d_{1;l}^\top \tilde v_t^{(2)} + d_{1;r}^\top \tilde v_t^{(1)}) d_{2;l}^\top \tilde v^{(2)}
		\end{alignedat}
	\end{equation}
	where
	\begin{equation}
		E = b^{(1)} \snorm{v^{(1)}_t}_H + b^{(2)} \snorm{v^{(2)}_t}_H + a^{(1)} (\tilde v^{(1)})^\top N \tilde v^{(1)} + a^{(2)} (\tilde v^{(2)})^\top N \tilde v^{(2)},
	\end{equation}
	defines an energy. Since $L \tilde w = L P v = 0$ we have $e_r^\top \tilde v^{(1)} = e_l^\top \tilde v^{(2)}$ and $d_{1;r}^\top \tilde v^{(1)} = -d_{1;l}^\top \tilde v^{(2)}$. Substituted into \eqref{eq: disc_syst_projhyb_enest} gives
	\begin{equation}
		\der{}{t} E = 0.
	\end{equation}
\end{proof}
\section{Time integration}
\label{sec: time_int}
All methods considered in this paper lead to a system of ODE:s (with $m$ and $2m$ unknowns for the single- and dual-block problems respectively) on the form
\begin{equation}
	\begin{aligned}
		v_{tt} & = D v, \quad &  & t \geq 0 \\
		v(t)   & = f_1,       &  & t = 0,   \\
		v_t(t) & = f_2,       &  & t = 0,
	\end{aligned}
\end{equation}
where $D$ is a matrix with real non-positive eigenvalues. To integrate the system in time we use a 4th order accurate method given by
\begin{equation}
	\begin{aligned}
		v^{(0)}     & = f_1,                                                     \\
		v^{(1)}     & = (I + \frac{k^2}{2} D) f_1 + k (I + \frac{k^2}{6} D) f_2, \\
		v^{(n+1)} & = (2I + k^2 D + \frac{k^4}{12} D D) v^{(n)} - v^{(n-1)},
	\end{aligned}
\end{equation}
where $k$ is the temporal step size. See \cite{Mattsson2006,Mattsson2015} for more details on this type of time-stepping methods. It is easily shown that stability is obtained if $k$ satisfies
\begin{equation}
	k^2 \rho(D) < 12,
\end{equation}
where $\rho(D)$ is the spectral radius of $D$. Introducing the undivided matrix $\tilde D = h^4 D$ we get the CFL condition
\begin{equation}
	k < \sqrt{\frac{12}{\rho(\tilde D)}} h^2.
\end{equation}
The spectral radius of the undivided operator $\tilde D$ is dependent on the spatial discretization but not on $h$ (for large enough problems). Thus, if the time step is chosen as a fraction of the stability limit, the theoretical execution time for a given problem size is proportional to $\sqrt{\rho(\tilde D)}$.
\section{Numerical experiments}
\label{sec: num_exp}
In this section we numerically verify and compare the methods. We evaluate the methods in terms of error convergence and theoretical execution time measured as the spectral radius of the undivided spatial operator (as outlined in Section \ref{sec: time_int}).
\subsection{Analytical solutions}
To measure the error and convergence of the methods we use standing waves solutions derived using separation of variables. The form of the analytical solutions are given by
\begin{equation}
	\label{eq: an_sol_form}
	u(x,t) = T(t) X(x),
\end{equation}
where
\begin{equation}
	T(t) = \cos(\beta^2 \sqrt{\frac{a}{b}} t) \quad \text{and}
\end{equation}
\begin{equation}
	X(x) = A_1 \cosh(\beta x) + A_ 2\cos(\beta x) + A_3 \sin(\beta x) + A_4 \sinh(\beta x).
\end{equation}
The constants $A_{1,2,3,4}$ and $\beta$ are found by imposing the boundary or interface conditions and solving the resulting non-linear system of equations numerically. For the multiblock problem the solutions in each block are given by \eqref{eq: an_sol_form} with separate parameters.

To evaluate the boundary treatments homogeneous beams with clamped and free homogeneous boundary conditions are considered. For these problems we choose $b = a = 1$ and the domain $x \in [0,1]$. The interface treatments are evaluated by considering a periodic domain $x \in [-1,1]$ with a parameter discontinuity at $x = 0$ and $x = \pm 1$ (two interface couplings). The parameters are chosen to $a^{(1)} = 1$, $a^{(2)} = 4$ and $b^{(1)} = b^{(2)} = 1$. The solutions at $t = 0$ (initial data) are plotted in Figure \ref{fig: init_data}. In \ref{appendix: an_sols} the parameters $A_{1,2,3,4}$ and $\beta$ for the homogeneous beam with clamped and free boundary conditions and parameters $A^{(1,2)}_{1,2,3,4}$ and $\beta^{(1,2)}$ for the piecewise homogeneous beam are presented.

The error is calculated as
\begin{equation}
	\epsilon^{(m)} = \norm{u - v^{(m)}}_h,
\end{equation}
where $u$ is the analytical solution, $v^{(m)}$ the approximate solution with $m$ grid points and $\norm{\cdot}_h$ the discrete $L^2$-norm. The time step is chosen as the stability requirement divided by two. The errors are measured at $t = 1$.
\begin{figure}
	\centering
	\includegraphics[width=1\textwidth]{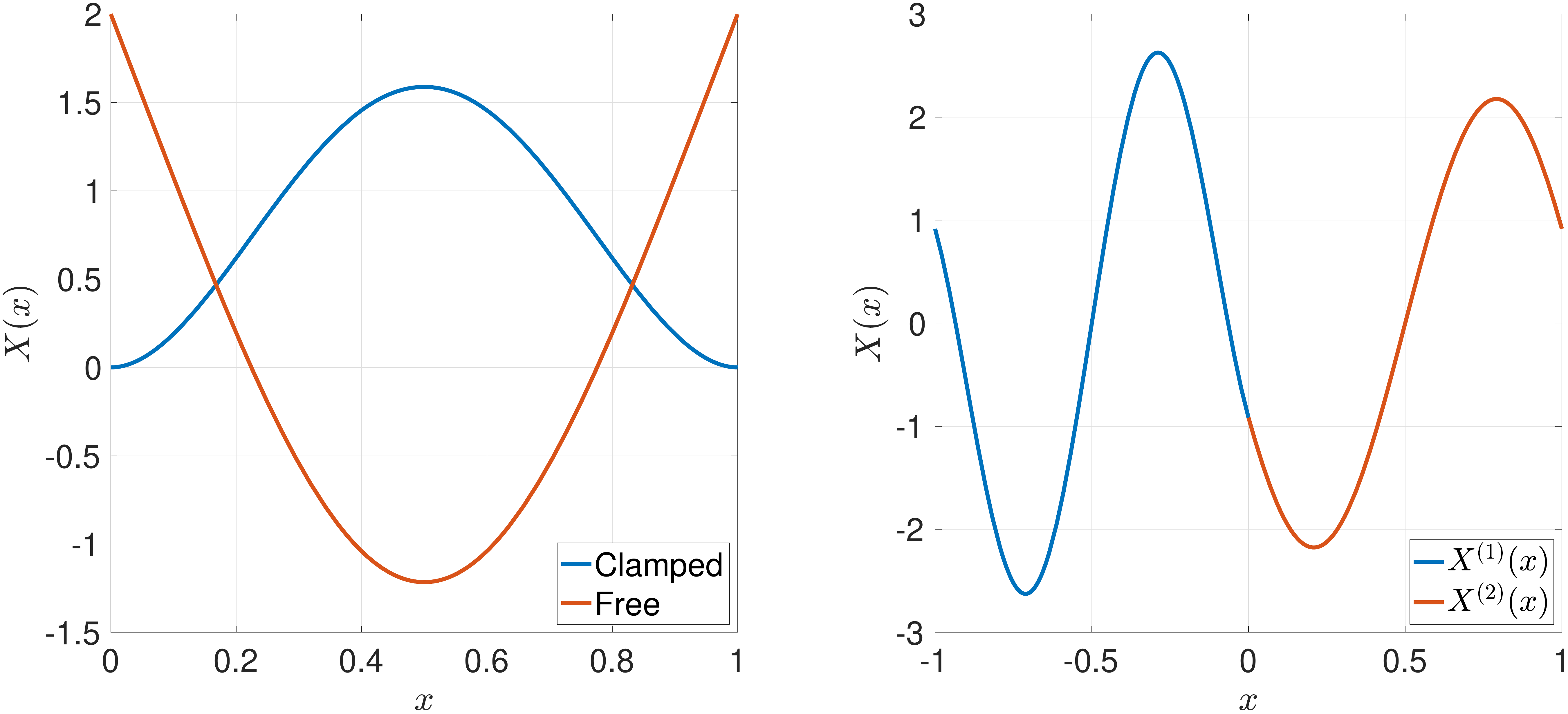}
	\caption{Left: Initial functions $X(x)$ for the homogeneous beam problem with clamped and free boundary conditions. Right: Initial functions $X^{(1,2)}(x)$ for the piecewise homogeneous beam problem with interface conditions at $x = 0$ and $x = \pm 1$.}
	\label{fig: init_data}
\end{figure}
\subsection{Choice of \texorpdfstring{$\alpha_{II}$}{alpha2} and \texorpdfstring{$\alpha_{III}$}{alpha3} for SAT}
When imposing clamped boundary conditions with SAT or interface conditions with only SAT one has to decide how to choose the parameters $\alpha_{II}$ and $\alpha_{III}$, see Sections \ref{sec: sat_single_clamped} and \ref{sec: proj_multiblock_sat}. In this section we investigate the influence of these parameters on the error and spectral radius (theoretical execution time, see Section \ref{sec: time_int}) for the homogeneous DBE with clamped boundary conditions.
\begin{figure}[htbp]
	\begin{subfigure}{\textwidth}
		\centering
		\includegraphics[width=0.82\textwidth]{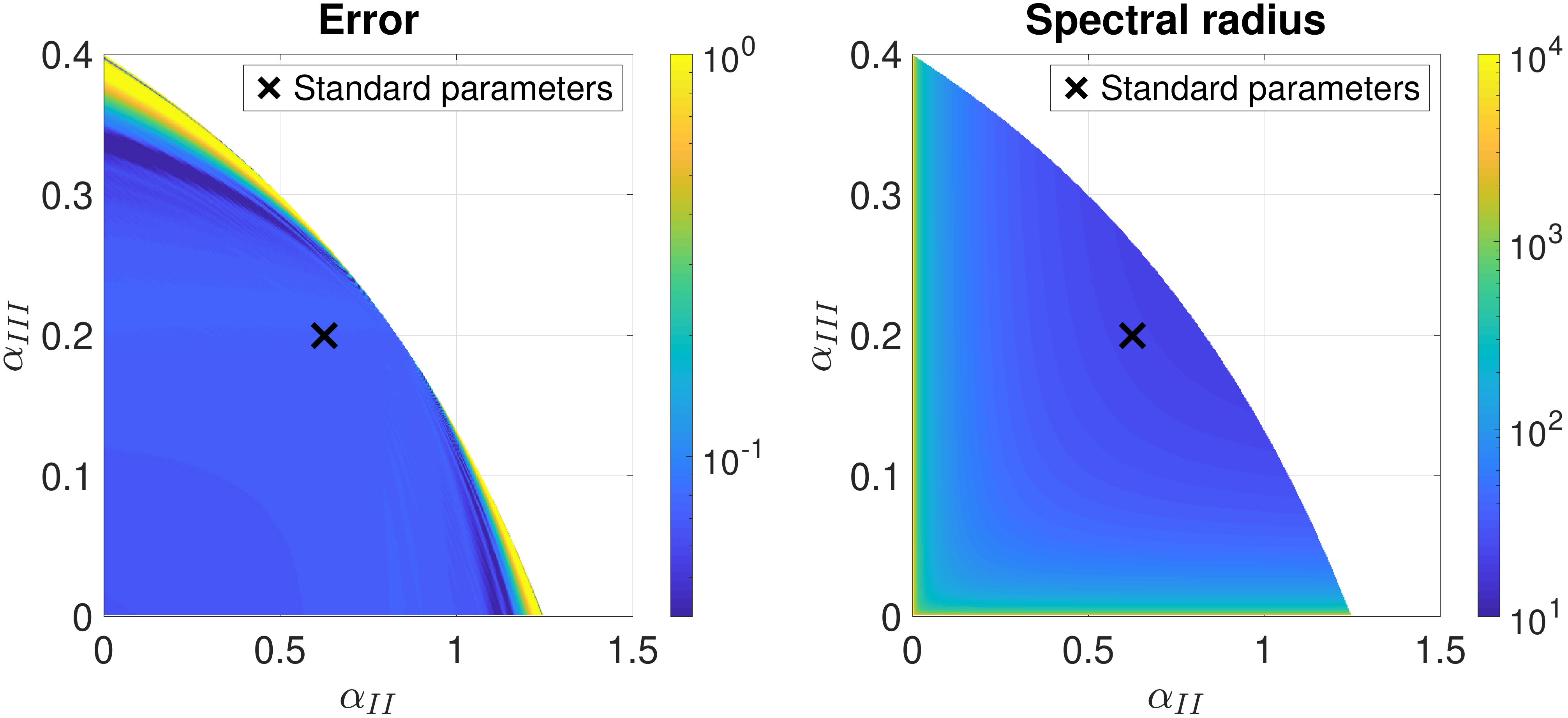}
		\caption{2nd order operators.}
		\label{fig: error_spec_alphas_o2}
	\end{subfigure}
	\newline
	\begin{subfigure}{\textwidth}
		\centering
		\includegraphics[width=0.82\textwidth]{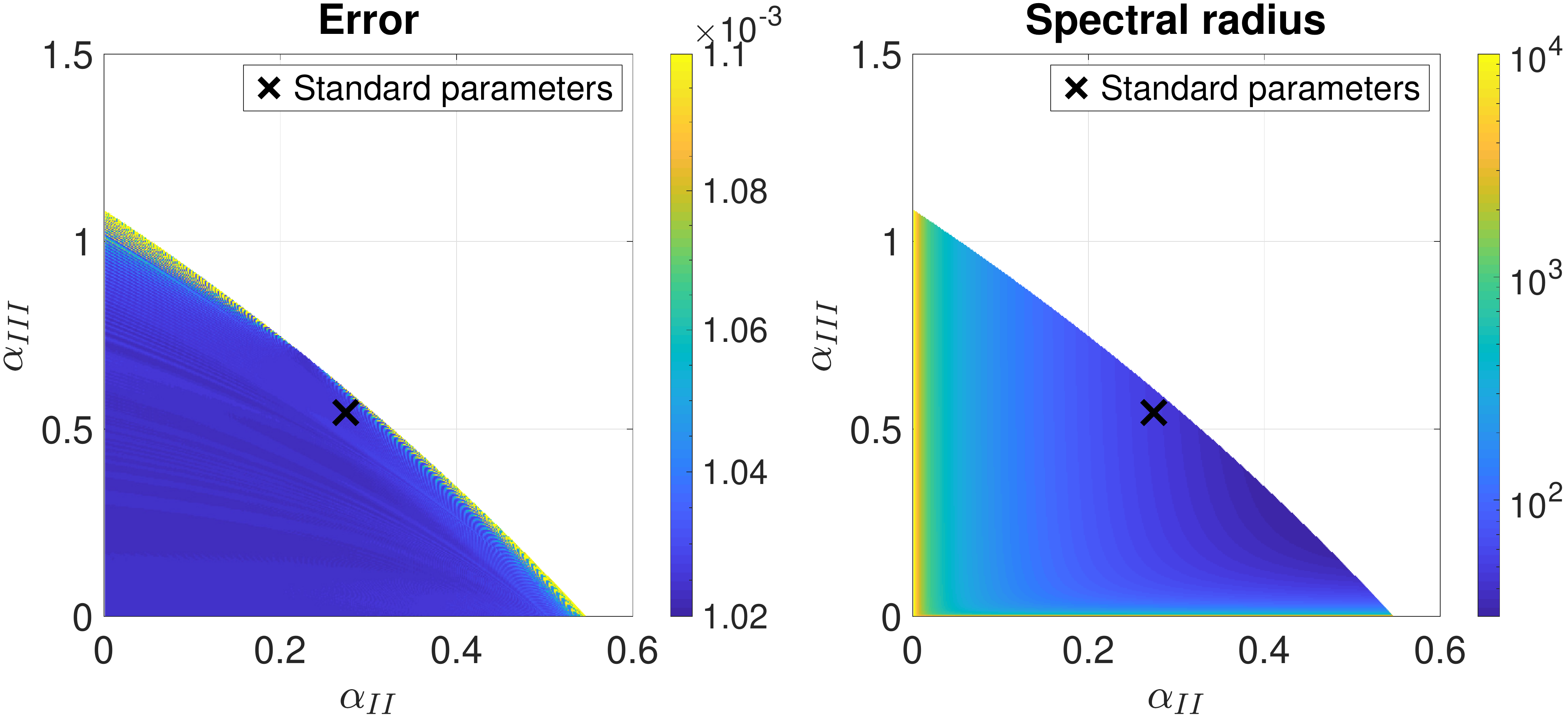}
		\caption{4th order operators.}
		\label{fig: error_spec_alphas_o4}
	\end{subfigure}
	\newline
	\begin{subfigure}{\textwidth}
		\centering
		\includegraphics[width=0.82\textwidth]{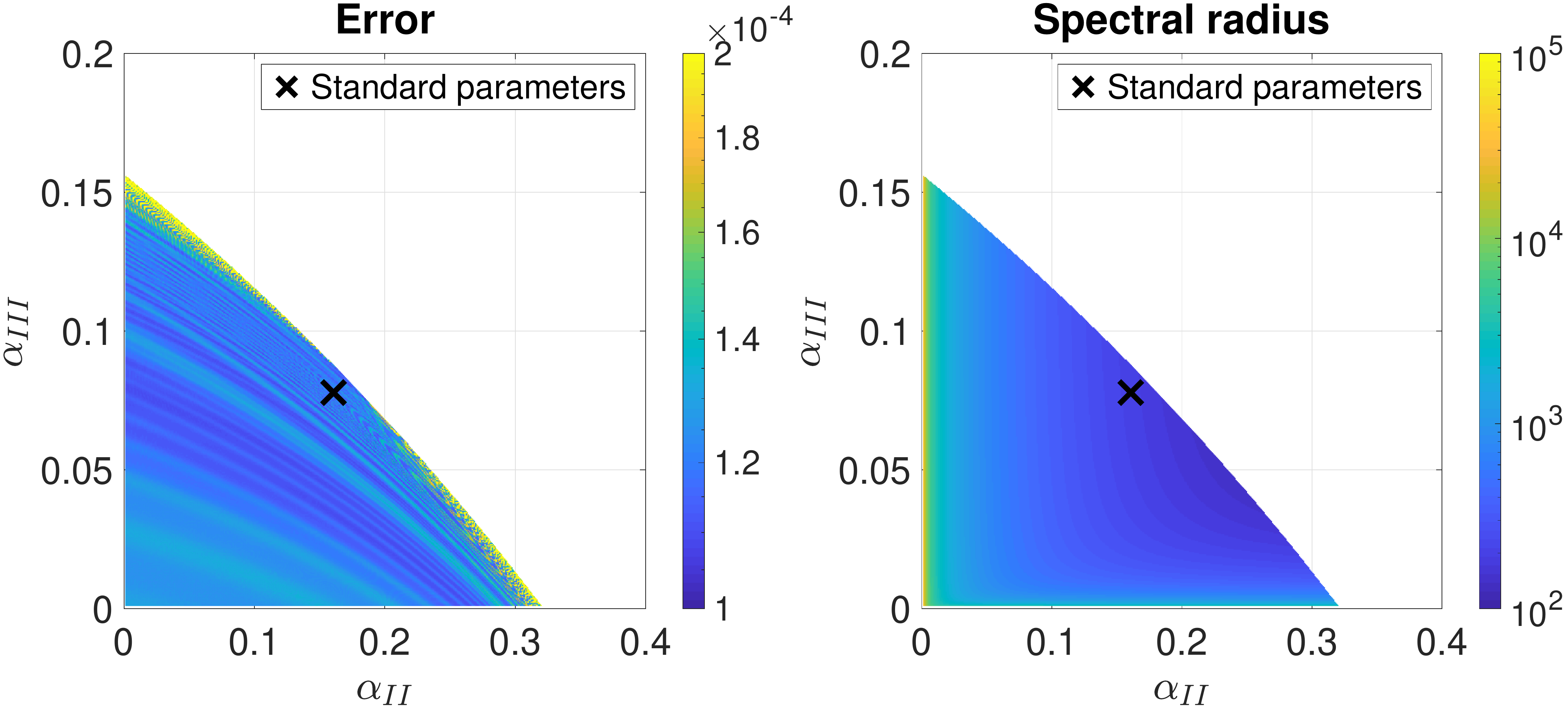}
		\caption{6th order operators.}
		\label{fig: error_spec_alphas_o6}
	\end{subfigure}
	\caption{Error and undivided spectral radius with stable $\alpha_{II}$ and $\alpha_{III}$ for the 2nd, 4th and 6th order SBP operators applied to homogeneous beam problem with clamped boundary conditions. The standard parameters are indicated.}
	\label{fig: error_spec_alphas}
\end{figure}

In Figure \ref{fig: error_spec_alphas} the error and undivided spectral radius for stable choices of $\alpha_{II}$ and $\alpha_{III}$ are plotted for the 2nd, 4th and 6th order operators with $m = 21$ . For all orders the spectral radius is greatest for small $\alpha_{II,III}$. This is expected since it leads to SAT with larger magnitude, and thus the boundary conditions are imposed more strongly.

As can be seen in the error plots of all orders, the smallest spectral radius leads to a relatively large error and, in some cases, the smallest error leads to the largest spectral radius. It is not clear whether or not an optimal choice exists. In this paper we use the parameters derived in \cite{Mattsson2015}, given by the largest $\alpha_{II}$ and $\alpha_{III}$ such that
\begin{equation}
	\begin{aligned}
		 & \frac{N}{2} - h \alpha_{II} (d_{2;l} d_{2;l}^\top + d_{2;r} d_{2;r}^\top) \geq 0 \quad \text{and} \\
		 & \frac{N}{2} - h^3 \alpha_{III} (d_{3;l} d_{3;l}^\top + d_{3;r} d_{3;r}^\top) \geq 0.
	\end{aligned}
\end{equation}
These parameters, here referred to as \emph{standard}, are determined numerically and presented in Table \ref{tbl: alphas} and Figure \ref{fig: error_spec_alphas}. Clearly, for this problem the standard parameters are not optimal. However, from extensive testing we have found that this choice leads to a reliable balance between error and spectral radius. Furthermore, computing the standard parameters is cheap and problem independent. Performing a parameter sweep such as the one used to generate Figure \ref{fig: error_spec_alphas} is highly expensive and not feasible for individual problems. In the present study we settle with using the standard parameters for the purpose of comparing SAT to projection. The question of optimal parameter choices is an interesting topic for future research.
\begin{table}
	\centering
	\caption{Standard $\alpha_{II}$ and $\alpha_{III}$ for the 2nd, 4th and 6th order operators.}
	\label{tbl: alphas}
	\begin{tabular}{ |c|c|c|c| }
		\hline
		Parameter      & 2nd order & 4th order & 6th   \\
		\hline
		$\alpha_{II}$  & 0.625     & 0.274     & 0.161 \\
		\hline
		$\alpha_{III}$ & 0.200     & 0.544     & 0.078 \\
		\hline
	\end{tabular}
\end{table}
\subsection{Spectral radius}
In this section we present and compare the spectral radii of the spatial operators for the homogeneous and piecewise homogeneous beam problems.

In Table \ref{tbl: hom_spectral_rad} the undivided spectral radii for the homogeneous beam problem with clamped and free boundary conditions imposed using SAT and projection are presented. The results show that the projection method leads to a smaller spectral radius, thus allowing for a larger time step, compared to SAT. Note also that the spectral radius with boundary conditions imposed using projection is independent of boundary conditions, whereas the clamped boundary conditions with SAT yields a significantly larger spectral radius compared to free boundary conditions with SAT.

In Table \ref{tbl: p_hom_spectral_rad} the undivided spectral radii for the piecewise homogeneous beam problem with interface conditions imposed using SAT, projection and the hybrid method are shown. For the 2nd and 4th order operators, the spectral radius is independent of method. For the 6th order operators the projection method yields the smallest spectral radius and SAT the largest.
\begin{table}[!htbp]
	\centering
	\caption{Spectral radius of undivided right-hand side matrix $\tilde D$ for homogeneous beam with clamped and free boundary conditions imposed using SAT and projection.}
	\label{tbl: hom_spectral_rad}
	\begin{tabular}{|c|c|c|c|c|}\hline
		BC                       & Method     & 2nd order & 4th order & 6th order \\
		\hline
		\multirow{2}{*}{Clamped} & SAT        & 22.4651   & 49.8208   & 202.8492  \\
		                         & Projection & 16.0000   & 26.6666   & 34.1333   \\
		\hline
		\multirow{2}{*}{Free}    & SAT        & 16.0000   & 28.3942   & 84.0057   \\
		                         & Projection & 16.0000   & 26.6666   & 34.1333   \\
		\hline
	\end{tabular}
\end{table}
\begin{table}[!htbp]
	\centering
	\caption{Spectral radius of undivided right-hand side matrix $\tilde D$ for piecewise homogeneous beam with interface conditions imposed using SAT, projection and hybrid.}
	\label{tbl: p_hom_spectral_rad}
	\begin{tabular}{|c|c|c|c|}\hline
		Method     & 2nd order & 4th order & 6th order \\
		\hline
		SAT        & 64.1945   & 106.6666  & 367.1694  \\
		Projection & 64.0000   & 106.6666  & 136.5332  \\
		Hybrid     & 64.0000   & 106.6666  & 193.7828  \\
		\hline
	\end{tabular}
\end{table}
\subsection{Error convergence}
\label{sec: results_convergence}
In this section we present accuracy and convergence results for the homogeneous and piecewise homogeneous beam problems. The theoretical rate of convergence for this problem is $\min(2p,r+4)$, where $2p$ is the internal stencil order and $r$ the boundary closure order \cite{Svard2019}. For the 2nd, 4th and 6th order SBP operators used here we have $p = 1,2,3$ and $r = -2, 0, 1$ (see \cite{Mattsson2014} for more details). Thus, the theoretical convergence rates are 2, 4 and 5 for the 2nd, 4th and 6th order SBP operators respectively.

In Figure \ref{fig: single_block_conv} the errors of the homogeneous beam problem with clamped and free boundary conditions imposed using SAT and projection are plotted for the 2nd, 4th and 6th order operators. The plots indicate good agreement with the theoretical convergence expectations for large step sizes. The roundoff errors are dominant with the 4th and 6th order operators for approximately $h < 5 \cdot 10^{-3}$. The results show that the error with SAT and projection are comparable for clamped boundary condition, and that SAT is slightly more accurate for free boundary conditions.

In Figure \ref{fig: multi_block_conv} we compare the errors for the piecewise homogeneous beam with interface conditions imposed using SAT, projection and the hybrid method. The results show that all three methods for imposing interface conditions are comparable in terms of error and that the expected convergence rates with all operators are obtained for large $h$.
\begin{figure}
	\centering
	\includegraphics[width=1\textwidth]{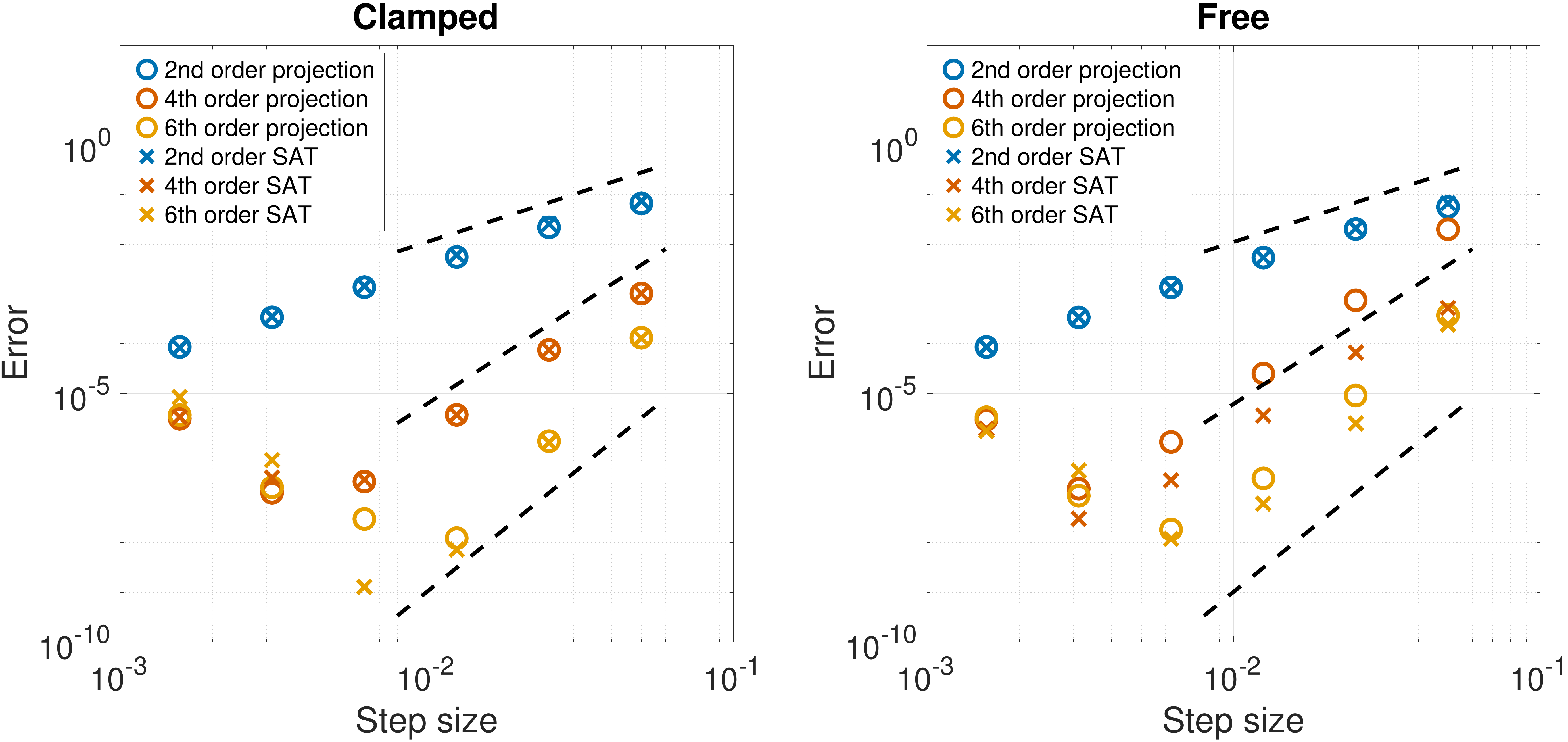}
	\caption{Error as a function of step size for a homogeneous beam with clamped and free boundary conditions imposed using SAT and projection. Results with 2nd, 4th and 6th order SBP operators are presented. The dashed lines indicate the theoretical convergence rates 2, 4 and 5.}
	\label{fig: single_block_conv}
\end{figure}
\begin{figure}
	\centering
	\includegraphics[width=0.48\textwidth]{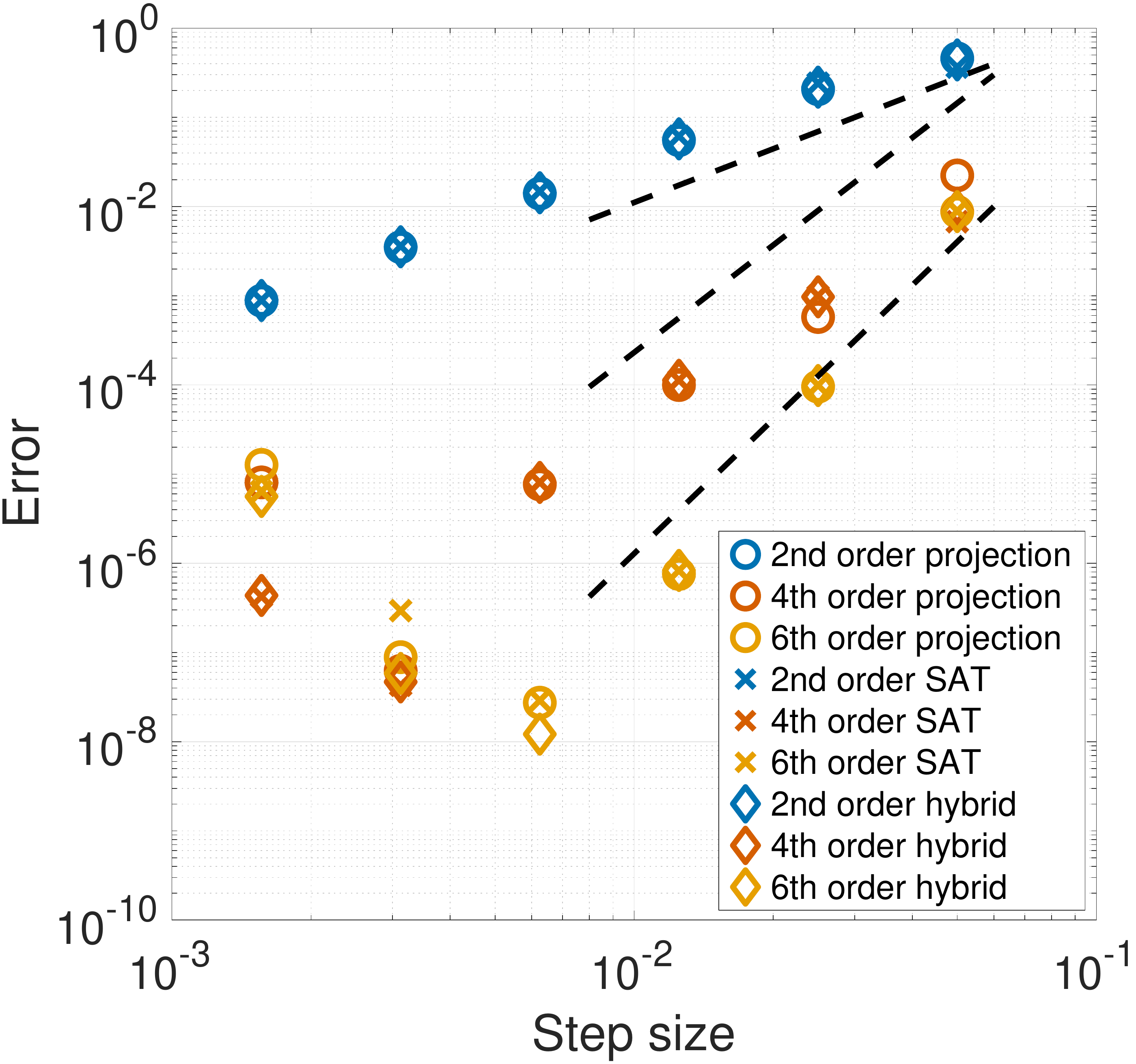}
	\caption{Error as a function of step size for circular piecewise homogeneous beam imposed using SAT, projection and hybrid method. Results with 2nd, 4th and 6th order SBP operators are presented. The dashed lines indicate the theoretical convergence rates 2, 4 and 5.}
	\label{fig: multi_block_conv}
\end{figure}
\section{Conclusions}
In this paper we have studied boundary treatments for the dynamic beam equation (DBE) using summation-by-parts finite differences (SBP-FD). We have compared the simultaneous approximation term (SAT) method to the projection method in terms of accuracy and efficiency for the DBE with clamped and free boundary conditions. We have also studied the DBE with piecewise constant parameters, where the discontinuities necessitate the use of internal boundaries. Novel SAT and projection methods for imposing the interface conditions are derived. Additionally, we present a novel hybrid method combining SAT and projection for interface conditions. The accuracy and efficiency of the methods are compared numerically.

Numerical experiments with specific SBP-FD operators show that SAT can be more accurate than projection for free boundary conditions, otherwise all methods considered are comparable in terms of accuracy. The projection method is found to yield the smallest spectral radius for all problems considered, thus leading to the least restrictive time step requirement for explicit time integration methods.

Awaiting further theoretical developments all methods are roughly equal and the choice between them is largely a matter of taste. However, we emphasize the difference between SAT and projection when proving stability of SBP discretizations. With projection the proofs closely mimic the continuous counterparts. One may argue that, for the problems considered here, the derivation of a semi-discrete energy estimate is redundant since it directly follows from the continuous proof if SBP-FD operators are used together with projection. With SAT the proofs are more involved, in particularly for clamped boundary conditions and interface conditions since these involve a decomposition of the fourth derivative operator and tuning of multiple parameters. The hybrid method for interface conditions avoids these problems but, as the numerical experiments show, it is not more accurate nor more efficient than the projection method.
\label{sec: concs}
\newpage
\appendix
\section{Parameters of analytical solutions}
\label{appendix: an_sols}

\begin{table}[!htbp]
	\centering
	\caption{Parameters for analytical solutions \eqref{eq: an_sol_form} of homogeneous beam with clamped and free boundary conditions.}
	\label{tbl: hom_params}
	\begin{tabular}{|c|c|c|}
		\hline
		        & Clamped            & Free               \\
		\hline
		$\beta$ & 4.730040744862704  & 4.730040744862704  \\
		$A_1$   & 1.0                & 1.0                \\
		$A_2$   & -1.0               & 1.0                \\
		$A_3$   & -0.982502214576238 & -0.982502214576238 \\
		$A_4$   & 0.982502214576238  & -0.982502214576238 \\
		\hline
	\end{tabular}
\end{table}
\begin{table}[!htbp]
	\centering
	\caption{Parameters for analytical solutions \eqref{eq: an_sol_form} of piecewise homogeneous beam coupled as a ring.}
	\label{tbl: inhom_params}
	\begin{tabular}{|c|c|}
		\hline
		              & Interface           \\
		\hline
		$\beta^{(1)}$ & -7.5615808319278868 \\
		$A_1^{(1)}$   & 0.6114706121627287  \\
		$A_2^{(1)}$   & -1.5276874719563509 \\
		$A_3^{(1)}$   & 2.0552837081134849  \\
		$A_4^{(1)}$   & -0.6121069366992447 \\
		\hline
		$\beta^{(2)}$ & -5.3468450827464249 \\
		$A_1^{(2)}$   & 0.0766810911329587  \\
		$A_2^{(2)}$   & -0.9928979509265810 \\
		$A_3^{(2)}$   & 1.9635450880282364  \\
		$A_4^{(2)}$   & 0.0774150750075980  \\
		\hline
	\end{tabular}
\end{table}

\newpage
\bibliographystyle{elsarticle-num}
\bibliography{references}

\end{document}